\documentclass[a4paper]{amsart}
\usepackage{amssymb}
\usepackage{stmaryrd}
\usepackage[all]{xy}
\usepackage{epsf}
\usepackage{enumerate}
\newcommand{\printname}[1] {}
\newcommand{\labell}[1] {\label{#1}\printname{#1}}
\newtheorem{theorem}{Theorem}
\newtheorem{proposition}{Proposition}
\newtheorem{lemma}{Lemma}
\newtheorem{corollary}{Corollary}

\theoremstyle{remark}
\newtheorem{remark}{Remark}

\newcommand{\Rr}{\mathbb R}

\newcommand{\X}{\ensuremath{\mathfrak{X}}}

\newcommand{\F}{\ensuremath{\mathcal{F}}}

\newcommand{\tto}{\rightrightarrows}

\newcommand{\G}{\mathcal{G}}            
\newcommand{\al}{\alpha}                
\newcommand{\be}{\beta}                 
\newcommand{\Lie}{\mathcal{L}}          
\renewcommand{\gg}{\mathfrak{g}}        

\renewcommand{\d}{\mathrm{d}}

\DeclareMathOperator{\Hol}{Hol}

\newcommand{\comment}[1]{}

\renewcommand{\labell}{\label} 

\begin{document}
\title{A geometric approach to Conn's linearization theorem}

\dedicatory{Dedicated to Alan Weinstein}
\author{Marius Crainic}
\address{Depart. of Math., Utrecht University, 3508 TA Utrecht, 
The Netherlands}
\email{crainic@math.uu.nl}

\author{Rui Loja Fernandes}
\address{Depart.~de Matem\'{a}tica, 
Instituto Superior T\'{e}cnico, 1049-001 Lisboa, PORTUGAL} 
\email{rfern@math.ist.utl.pt}

\thanks{MC was supported in part by NWO and a Miller Research Fellowship. RLF was supported in part by FCT/POCTI/FEDER and grants
  POCI/MAT/55958/2004 and POCI/MAT/57888/2004.}


\begin{abstract}
We give a soft geometric proof of the classical result due to Conn
stating that a Poisson structure is linearizable around a singular point (zero)
at which the isotropy  Lie algebra is compact 
and semisimple.
\end{abstract}
\maketitle

\setcounter{tocdepth}{1}

\section*{Introduction}             %
\labell{Linearization: Outline}    %

Recall that a \textbf{Poisson bracket} on a manifold $M$ is a Lie bracket
$\{\cdot,\cdot\}$ on the space $C^{\infty}(M)$ of smooth functions on
$M$, satisfying the derivation property
\[ \{ fg, h\}= f\{ g, h\}+ g\{ f, h\}, \quad f,g,h \in
C^{\infty}(M).\]
Let us fix a \textbf{zero} of the Poisson bracket, i.e., a point $x_0\in M$ where 
$\{f,g\}(x_0)=0$, for all functions $f,g\in C^{\infty}(M)$. Then $T_{x_0}^*M$  becomes a 
Lie algebra with the Lie bracket:
\[ [\d_{x_0} f,\d_{x_0} g]:=\d_{x_0}\{f,g\}.\] 
This Lie algebra is called the \textbf{isotropy Lie algebra} at $x_0$ and will
be denoted by $\gg_{x_0}$. Equivalently, the tangent space $T_{x_0}M=\gg_{x_0}^*$ 
carries a canonical linear Poisson bracket called the \textbf{linear approximation}
at $x_0$. The \emph{linearization problem} for $(M,\{\cdot,\cdot\})$ around $x_0$ is the following:
\begin{itemize}
\item Is there a Poisson diffeomorphism $\phi:U\to V$ from a neighborhood 
$U\subset M$ of $x_0$ to a neighborhood $V\subset T_{x_0}M$ of $0$?
\end{itemize}
When $\phi$ exists, one says that the Poisson structure is \textbf{linearizable}
around $x_0$. The most deep linearization result is the following theorem
due to Conn \cite{Conn2}:

\begin{theorem}
\labell{thm:main}
Let $(M,\{\cdot,\cdot\})$ be a Poisson manifold with a zero $x_0\in M$. If the isotropy
Lie algebra $\gg_{x_0}$ is semisimple of compact type, then $\{\cdot,\cdot\}$ is linearizable 
around $x_0$.
\end{theorem}

Note that there exists a simple well-known criterion to decide if $\gg_{x_0}$ is semisimple of compact type: its \emph{Killing form} $K$ must be negative definite.

The proof given by Conn in \cite{Conn2} is analytic. He uses a
combination of Newton's method with smoothing operators, as devised by
Nash and Moser, to construct a converging change of
coordinates. This proof is full of difficult estimates and, in spite of
several attempts to find a more geometric argument, it is 
the only one available up to now. See, also, the historical comments at
the end of this paper.

In this paper we will give a soft geometric proof of this result using Moser's path
method. At the heart of our proof is an integration argument and an averaging argument. The averaging enters
into the proof in a similar fashion to the proofs of other linearization theorems, such as Bochner's Linearization Theorem for actions of compact Lie group around fixed points. Our proof gives a new geometric insight to the theorem, clarifies the compactness assumption, and should also work in various other situations. More precisely, the proof consists of the following four steps:
\begin{description}
\item[Step 1] \emph{Moser's path method.} Using a Moser's path method, we prove a \emph{Poisson version} of Moser's theorem (see Theorem \ref{thm:Moser}), which is inspired by the work of Ginzburg and Weinstein \cite{GiWe}. It reduces the proof of Conn's Theorem to showing that the 2nd Poisson cohomology around $x_0$ vanishes.
\item[Step 2] \emph{Reduction to integrability around a fixed point.} Using the vanishing of cohomology for proper Lie groupoids and the general Van Est theorem relating groupoid and algebroid cohomology \cite{Cra}, we show that it is enough to prove integrability of the Poisson structure around a fixed point $x_0$.
\item[Step 3] \emph{Reduction to the existence of symplectic realizations.} Using the equivalence of integrability in the Poisson case and the existence of complete symplectic realizations \cite{CrFe2}, we show that it is enough to construct a symplectic realization of a neighborhood of $x_0$ with the property that the fiber over $x_0$ is 1-connected and compact.
\item[Step 4] \emph{Existence of symplectic realizations.} The same path space used in \cite{CrFe1} to determine the precise obstructions to integrate a Lie algebroid and to explicitly construct an integrating Lie groupoid, yields that a neighborhood of $x_0$ admits the desired symplectic realization.
\end{description}

The fact that the tools that we use only became available recently probably explains why it took more than 20 years to find a geometric proof of Conn's Theorem.

The four sections that follow describe each of the steps in the proof. We conclude the paper with two appendices: the first one contains an auxiliary proposition on foliations (which is used in the last step), while in the second one make some historical remarks.

Finally, we would like to mention that our method works in other situations as well. A similar linearization result around symplectic leaves instead of fixed points
is being worked out in \cite{Ionut}. The analogue of Conn's Theorem for Lie algebroids (conjectured in \cite{Wein2} and proved in \cite{MonZun}) can also be proved by our method, the only missing step being the proof of the vanishing conjecture of \cite{CrMo} (one must replace the Poisson cohomology of Step 1 by the deformation cohomology of \cite{CrMo}). Details will be given elsewhere. 
It would also be interesting to find a similar geometric proof of the smooth Levi decomposition theorem of Monnier 
and Zung \cite{MonZun}.

\section*{Step 1: Moser's path method}        %

Let us start by recalling that a Poisson bracket $\{\cdot,\cdot\}$ on $M$ can
also be viewed as a bivector field $\pi\in \Gamma(\wedge^2 TM)$ with
zero Schouten bracket $[\pi,\pi]=0$. One determines the other through
the relation  
\[ \pi(\d f\wedge \d g)=\{f,g\},\quad (f,g\in C^{\infty}(M)).\] 
Recall also, that the \textbf{Poisson cohomology} of $M$ (with trivial coefficients) is the cohomology of the complex
$(\X^k(M),\d_\pi)$, where $\X^k(M)$ is the space of $k$-vector fields,
and the differential is defined by 
\[ \d_\pi \theta:=[\pi,\theta].\]

When $x_0$ is a zero of $\pi$, we can consider the \textbf{local Poisson cohomology groups} $H_{\pi}^k(M,x_0)$. By this we mean the Poisson cohomology group of the germ of $(M,\pi)$ at $x_0$, i.e., the group $\varinjlim H_{\pi}^k(U)$ obtained by taking the direct limit of the ordinary Poisson cohomology groups of $U$, when $U$ runs over the filter of open neighborhoods of $x_0$.
\vskip 10 pt

\begin{theorem}
\labell{thm:Moser}
Let $(M,\{\cdot,\cdot\})$ be a Poisson manifold with a zero $x_0$.
Assume that the Lie algebra cohomology groups $H^1(\gg_{x_0})$ and $H^1(\gg_{x_0}, \gg_{x_0}^*)$ vanish.
If $H^{2}_{\pi}(M, x_0)= 0$, then $\{\cdot,\cdot\}$ is linearizable at $x_0$.
\end{theorem}

For the proof, we will apply a Poisson version of Moser's path method. Since this is a local result, we can assume that $M=\Rr^m$ and $x_0=0$. Also, to simplify the notation we denote by $\gg$ the isotropy Lie algebra at $0$. We consider the path of Poisson structures $\pi_t$ on $\Rr^m$ defined by the formula
\[ \pi_t(x)=\frac{1}{t}\pi(tx), \quad (t\in [0,1]).\]
Then $\pi_1=\pi$, while $\pi_0=\pi^\text{lin}$ is the linearization of
$\pi$ at the origin. Moser's method will give us an isotopy
$\{\phi_t\}$, $0\le t\le 1$, defined in a neighborhood of the origin,
and such that
\[ (\phi_t)_*\pi_t=\pi_0, \quad (t\in [0,1]).\]
Therefore $\phi_1$ will be the desired linearization map.
To construct $\phi_t$ let us consider the bivector field
$\dot{\pi}_t:=\frac{\d\pi_t}{\d t}$. 

\begin{lemma}
There exists a vector field $X$ around the origin $0\in \Rr^n$ such that
\begin{equation}
\labell{eq:deform}
\Lie_X \pi=-\dot{\pi}_1,
\end{equation}
and $X$ has zero linearization at the origin.
\end{lemma}

\begin{proof}
Differentiating the equation $[\pi_t,\pi_t]=0$ with respect to $t$, we
obtain
\[ \d_\pi\dot{\pi}_1=[\pi,\dot{\pi}_1]=0,\]
so $\dot{\pi}_1$ is a Poisson 2-cocycle. Hence its restriction to a
ball around the origin will be exact, i.e., we find a vector field $Y$
on the ball such that
\[ \dot{\pi}_1=\d_\pi Y=-\Lie_Y\pi.\]
This relation has two consequences:
\begin{enumerate}[(a)]
\item Since $\dot{\pi}_1$ vanishes at $x_0$, if we evaluate both
  sides on a pair of 1-forms and set $x=0$, we see that
  $Y_{x_0}([\al,\be])=0$, for $\al,\be\in\gg_{x_0}$.  Since $H^1(\gg)=0$ (i.e.
  $[\gg, \gg]= \gg$), we conclude that $Y_{x_0}=0$. Let $Y_0$ be the
  linearization of $Y$ at the origin.
\item Since $\dot{\pi}_1$ has zero linearization at the origin,
  if consider the linear terms of both sides at $x_0$ we obtain
  \[ \d_{\pi_0}Y_0=0.\]
\end{enumerate}
Hence, $Y_0$ is a 1-cocycle for the coadjoint representation of
$\gg_x$. Since $H^1(\gg, \gg^*)=0$, $Y_0$ must be a coboundary, so
there exists $v\in\gg$ such that $Y_0=\d_{\pi_0}v$.
The vector field $X=Y-\d_{\pi}v$ satisfies (\ref{eq:deform}) and has
zero linearization at the origin.
\end{proof}

\begin{proof}[Proof of Theorem \ref{thm:Moser}]
If $X$ is a vector field as in the Lemma, consider the time-dependent
vector field $X_t(x):=\frac{1}{t^2}X(tx)$. From (\ref{eq:deform}) we
obtain immediately that
\[ \Lie_{X_t} \pi_t=-\dot{\pi}_t.\]
Let $\phi_t$ be the flow of $X_t$. Since $X_t(0)=0$, we see that
$\phi_t$ is defined in some neighborhood $V$ of the origin for $0\le
t\le 1$. Also, we compute:
\[
\frac{\d}{\d t}(\phi_t)_*\pi_t=
(\phi_t)_*\left(\Lie_{X_t}\pi_t+\frac{\d\pi_t}{\d t}\right)=0.
\]
We conclude that $\phi_t$ is a diffeomorphism of $V$ with the desired property.
\end{proof}

\section*{Step 2: Reduction to integrability around a fixed point}%

In this section we explain the statement and we prove the following proposition which,
when combined with Theorem \ref{thm:Moser}, reduces the
proof of Conn's Theorem to integrability around a fixed point:


\begin{proposition}
\labell{prop:integrable}
Let $(M,\pi)$ be a Poisson manifold, $x_0\in M$ a fixed point. If 
some neighborhood of $x_0$ is integrable by a Hausdorff
Lie groupoid with 1-connected s-fibers, then 
\[ H^{2}_{\pi}(M, x_0)= 0.\]
More precisely, 
\begin{enumerate}[(i)]
\item There exist arbitrarily small
neighborhoods $V$ of $x_0$ which are integrable by Hausdorff proper 
groupoids $\G\tto V$ with homological $2$-connected fibers. 
\item For any such $V$, $H^{1}_{\pi}(V)=H^{2}_{\pi}(V)= 0$.
\end{enumerate}
\end{proposition}


The geometric object behind the Poisson brackets which provides the
bridge between Poisson geometry and Lie-group type techniques is the
\textbf{cotangent Lie algebroid} $A= T^*M$ and the associated groupoid
$\G(A)$ (see \cite{CrFe1,CrFe2}). For a Poisson manifold $M$ we will denote by
$\Sigma(M, \pi)=\G(T^*M)$ its associated groupoid. We recall that
$\Sigma(M)$ is defined as the set of cotangent paths in $M$ modulo
cotangent homotopies, and that it is a \emph{topological groupoid} with
1-simply connected s-fibers. A Poisson manifold $M$ is said to be 
\textbf{integrable} if the associated Lie algebroid $T^*M$ is
integrable. This happens iff $\Sigma(M, \pi)$ is a \emph{Lie groupoid}. In this
case, $\Sigma(M, \pi)$ carries a natural symplectic structure, that makes
it into a \textbf{symplectic groupoid}.

\begin{proof}[Proof of Proposition \ref{prop:integrable}] Let's assume that $U$ is a neighborhood of
$x_0$ which is integrable by an Hausdorff Lie groupoid $\G\tto U$. The
fiber of the source map $s:\G\to U$ above $x_0$ is a Lie group
integrating $\gg_{x_0}$, so it is compact and 1-connected. Hence, by
Reeb stability, there exists a neighborhood $V_0$ of $x_0$ such that
$s^{-1}(V_0)$ is diffeomorphic to the product $V_0\times G$. If we let
$V= t(s^{-1}(V_0))\subset U$ be the saturation of $V_0$, the
restriction $\G_{V}$ of $\G$ to $V$ will be a groupoid whose source map
has compact, 1-connected, fibers: using right translations, each fiber
will be diffeomorphic to $s^{-1}(x_0)\simeq G$. Moreover,
a compact Lie group has the same rational homology type as a product
of odd dimensional spheres, so $G$ is automatically homological
$2$-connected, so the $s$-fibers are also homological
$2$-connected.

The proof of the second part is a combination of two classical results on Lie groups
which have been extended to Lie groupoids. 
The first result states that the differentiable cohomology (defined
using groups cocycles which are smooth) vanishes for compact groups,
and this follows immediately by averaging. This result immediately
extends to groupoids, i.e. $H^{*}_{\text{diff}}(\G)=0$ for
any proper groupoid $\G$ (\cite{Cra}, Proposition 1).

The second result is the Van Est isomorphism. As explained in
\cite{Cra}, differentiable group(oid) cocycles can be differentiated
and they give rise to Lie algebra(oid) cocycles. The resulting map
$\Phi: H^{k}_{\text{diff}}(\G)\to H^{k}(A)$, called also the Van Est map, is
an isomorphism for degree $k\le n$ provided the $s$-fibers of $\G$ are
homological $n$-connected (\cite{Cra}, Theorem 3). Again, the proof is just an
extension of the classical proof of Van Est.

If we apply these two results to our groupoid $\G\tto V$, the second part of the proposition follows
since the Poisson cohomology of $V$
coincides with the Lie algebroid cohomology of $A=T^*V$.
\end{proof}

\section*{Step 3: Reduction to the existence of symplectic realizations}         %

In the previous step, we have reduced the proof of Conn's Theorem to integrability around a fixed point.
The integrability of a Poisson manifold $(M,\pi)$ is strongly related to the existence
of symplectic realizations. 

Recall that a symplectic realization of $(M, \pi)$ consists of
a symplectic manifold $S$ together with a Poisson map $\nu: S\to M$ which is 
a surjective submersion. One calls it complete if for any complete Hamiltonian vector field $X_{f}$ 
on $M$, the vector field $X_{\nu^*(f)}$ is complete. It is known that the existence of 
complete symplectic realizations is equivalent to integrability (Theorem 8 in \cite{CrFe2}),
but that depends on subtleties regarding the (required) Hausdorffness conditions on $S$
which are not relevant for us since we are interested on Hausdorff Lie groupoids.
Instead, in this paper we do require $S$ to be Hausdorff and we extract from \cite{CrFe2} the following result. In the statement we use the following conventions: for a symplectic realization $\nu: S\to M$ we denote by $\mathcal{F}(\nu)$ the foliation of $S$ by the (connected components of the) fibers of $\nu$, and $\mathcal{F}(\nu)^{\perp}$ is its symplectic orthogonal. Also, we recall that a foliation is simple if it is induced by a submersion.

\begin{theorem} 
\label{thm:integr:real}
A Poisson manifold $(M, \pi)$ is integrable by a Hausdorff Lie groupoid with 
1-connected s-fibers if and only if it admits a complete symplectic realization
$\nu: S\to M$ with the property that the foliation $\mathcal{F}(\nu)^{\perp}$
is simple and has simply-connected leaves. 
\end{theorem}

\begin{proof} 
One direction is clear: the source map of a Lie groupoid as in the
statement provides the desired symplectic integration (the symplectic orthogonals of the
s-fibers are the t-fibers).  Assume now that $\nu: S\to M$ is a symplectic integration as 
in the statement. Theorem 8 in \cite{CrFe2} insures that 
$\Sigma= \Sigma(M, \pi)$ is smooth (but possibly non-Hausdorff). A simple remark on the proof
of the cited theorem implies that, under our hypothesis,  $\Sigma$ is actually Hausdorff. 
Recall the main steps of the proof : 
the assignment $X_{f}\mapsto X_{\nu^*(f)}$ induces an action of the
Lie algebroid $T^*M$ on $S$ which integrates to an action of the Lie groupoid $\Sigma$ on $S$
and that the associated action groupoid is homeomorphic to the monodromy groupoid of $\mathcal{F}(\nu)^{\perp}$,
which we denote by $\G(\mathcal{F}^{\perp})$. In other words, we have
\[ \Sigma\times_{M}S \cong \G(\mathcal{F}^{\perp}).\]
where the fibered product is over $s$ and $\nu$. Since the right hand side is smooth, it follows easily \cite{CrFe2} that
$\Sigma$ is smooth as well and the previous homeomorphism is a diffeomorphism. Finally, note that $\mathcal{F}^{\perp}$ is 
induced by a submersion $\pi: S\to B$, for some manifold $B$, and that its leaves are simply connected. Therefore, we see
that $\G(\mathcal{F}^{\perp})= S\times_{B} S$ is Hausdorff. We conclude that $\Sigma$ is Hausdorff as well. 
\end{proof}

\begin{remark} 
The proof actually shows that the conditions on $\mathcal{F}^{\perp}$ can be replaced
by the fact that it has no vanishing cycles.
\end{remark}

The following corollary reduces the proof of Conn's Theorem to the existence of symplectic realizations around a fixed point:

\begin{corollary}
\label{realizations} 
Let $(M, \pi)$ be a Poisson manifold, $x_0\in M$ a fixed point and assume that
a neighborhood $U$ of $x_0$ admits a symplectic realization $\nu: S\to U$ with the property that $\nu^{-1}(x_0)$ 
is simply connected and compact. Then there exists a neighborhood of $x_0$ which is integrable by a Hausdorff Lie groupoid with 1-connected s-fibers.
\end{corollary}

\begin{proof} Note that $\nu^{-1}(x_0)$ is a Lagrangian submanifold of $S$. Therefore, $\nu^{-1}(x_0)$ is a compact, 1-connected, leaf of $\mathcal{F}^{\perp}(\nu)$. By Reeb stability, nearby leaves are compact, 1-connected and $\mathcal{F}^{\perp}(\nu)$ is simple. Hence we can apply Theorem \ref{thm:integr:real}.
\end{proof}

\section*{Step 4: Existence of symplectic realizations}%

The proof of Conn's Theorem can now be concluded by proving:

\begin{theorem} 
\label{thm:exist:real}
Let $(M,\pi)$ be a Poisson manifold, $x_0\in M$ a fixed 
point, and assume that the isotropy Lie algebra $\mathfrak{g}$ at $x_0$ is semi-simple of
compact type, with associated simply connected Lie group $G$. Then
there exists a symplectic realization $\nu: S\to U$ of some open neighborhood 
$U$ of $x_0$ such that $\nu^{-1}(x_0)= G$. 
\end{theorem}

We first recall some of the general properties of $\Sigma(M)$ (see \cite{CrFe1}). To construct it as a topological space
and possibly as a smooth manifold (in the integrable case), we consider the Banach manifold $P(T^*M)$ 
consisting of paths $a:I\to T^*M$ of class $C^2$, with
the topology of uniform convergence of a map together with its derivatives. Inside this Banach manifold
we have the space of cotangent paths:
\[ X:=\{a\in P(T^*M):\pi^\sharp(a(t))=\frac{\d}{\d t}p(a(t))\},\]
where $p:T^*M\to M$ is the bundle projection. Then $X$ is a submanifold of $P(T^*M)$ which carries 
a canonical foliation $\F$: two cotangent paths $a_0$ and $a_1$ belong to the same leaf if they are 
cotangent homotopic. This foliation has finite codimension and leaf space precisely $\Sigma(M)$.
Concatenation of paths, makes $\Sigma(M)$ into a topological groupoid which is smooth precisely 
when $M$ is integrable. 

The symplectic structure on $\Sigma(M)$ is a consequence of the following general property: the restriction of the canonical symplectic form of $P(T^*M)\simeq T^*P(M)$ to $X$ has kernel $\mathcal{F}$ and is invariant under the holonomy of $\mathcal{F}$. We conclude, also, that any transversal to $\mathcal{F}$ carries a symplectic structure invariant under the (induced) holonomy action. Therefore, the quotient of such a transversal by the holonomy action gives a symplectic manifold, provided the quotient is smooth. Unfortunately, achieving smoothness is difficult (and it would imply integrability directly). Instead, we will perform a quotient modulo only \emph{some} holonomy transformations, so that the result is smooth, and we will see that this is enough for our purposes.

\begin{proof}
First of all, we consider the source map $s:X\to M$ which sends a cotangent path $a(t)$ to its initial base point $p(a(0))$. This is a smooth submersion, and we look at the fiber $Y=s^{-1}(x_0)$. Since $x_0$ is a zero of $\pi$, $Y$ is saturated by leaves of $\F$ and we set $\F_Y=\F|_Y$. The quotient $G=Y/\F_Y$ is the 1-connected Lie group integrating the isotropy Lie algebra $\gg_{x_0}$, so it is compact. Moreover, note that we can canonically identify $Y$ with paths in the Lie group $G$ which start at the origin, so that the quotient map $Y\to Y/\F_Y=G$ sends a path to its end point. Also, two points in $Y$ belong to the same leaf of $\F_Y$ if the corresponding paths are homotopic relative to the end points. Since the first and second homotopy groups of $G$ vanish, the leaves of $\F_Y$ are 1-connected fibers of a 
locally trivial fibration $Y\to G$ with compact base, where local triviality is in the sense of 
Proposition \ref{technical} below (use right translations by contracting homotopies). By the proposition 
one can find:
\begin{enumerate}[(i)]
\item a transversal $T_X\subset X$ to the foliation $\mathcal{F}$ such that $T_Y:= Y\cap T_X$ is a complete transversal to $\mathcal{F}_Y$ (i.e., intersects each leaf of $\mathcal{F}_Y$ at least once).
\item a retraction $r: T_X\to T_Y$.
\item an action of the holonomy of $\mathcal{F}_Y$ on $r: T_X\to T_Y$ along $\mathcal{F}$.
\end{enumerate}
Moreover, the orbit space $S:=T_X/\Hol_{T_Y}(\mathcal{F}_Y)$
is a smooth (Hausdorff) manifold. Notice that the source map induces a map $\nu:S\to U$, where $U$ is an open neighborhood of $x_0$. Also, $\nu^{-1}(x_0)=Y/\F_Y=G$ is compact.
It follows that $S$ carries a symplectic form and that $\nu:S\to U$ is a Poisson map, so it satisfies all the properties in the statement of the theorem.
\end{proof}

\setcounter{section}{4}

\section*{Appendix 1: A technical result on foliations}    %

The aim of this section is to prove the following result which was used in the proof of Theorem \ref{thm:exist:real}.

\begin{proposition}
\label{technical} 
Let $\mathcal{F}$ be a foliation of finite codimension on a Banach manifold $X$ and let $Y\subset X$ be a submanifold which is saturated with respect to $\mathcal{F}$ (i.e., each leaf of $\mathcal{F}$ which hits $Y$ is contained in $Y$). Assume that:
\begin{itemize}
\item[(H1)] The foliation $\mathcal{F}_Y:= \mathcal{F}|_{Y}$ has simply connected leaves and is induced by a submersion $p: Y\to B$ into a compact manifold $B$.
\item[(H2)] The fibration $p: Y\to B$ is locally trivial, in the sense that its restriction to any contractible open set $U\subset B$ is trivial. 
\end{itemize}
Then one can find:
\begin{enumerate}[(i)]
\item a transversal $T_X\subset X$ to the foliation $\mathcal{F}$ such that $T_Y:= Y\cap T_X$ is a complete transversal to $\mathcal{F}_Y$ (i.e., intersects each leaf of $\mathcal{F}_Y$ at least once).
\item a retraction $r: T_X\to T_Y$.
\item an action of the holonomy of $\mathcal{F}_Y$ on $r: T_X\to T_Y$ along $\mathcal{F}$.
\end{enumerate}
Moreover, the orbit space $T_X/\Hol_{T_Y}(\mathcal{F}_Y)$ is
a smooth (Hausdorff) manifold.
\end{proposition}

\begin{remark}
In the previous proposition, by an action of the holonomy of $\mathcal{F}_Y$ on $r: T_X\to T_Y$ we mean that
an action of the holonomy groupoid of $\mathcal{F}_Y$ restricted to $T_Y$, denoted $\Hol_{T_Y}(\mathcal{F}_Y)$, 
on the map $r: T_X\to T_Y$ (recall that groupoids act on smooth maps over the space of units). Also, when we say ``along the leaves of $\mathcal{F}$'' we mean that the orbits of the action lie inside the leaves of $\mathcal{F}$. 

In the situation described by the proposition, $\mathcal{F}_Y$ is simple and this can be made more explicit in the following way. The action is given by a smooth family of diffeomorphisms $h_{x,y}:r^{-1}(x)\to r^{-1}(y)$ defined for any $x,y\in T_Y$ with $p(x)=p(y)$, satisfying $h_{y,z}\circ h_{x,y}=h_{x,z}$ and $h_{x,x}=I$. Also, the action being along the leaves of $\mathcal{F}$ means that $h_{x,y}(u)$ and $u$ are in the same leaf of $\mathcal{F}$, for any $u\in r^{-1}(x)$.
\end{remark}

\begin{remark}
Note also that the proposition is essentially of a finite dimensional nature (it is about the transversal geometry of a foliation
of finite codimension). Actually, using the language of \'etale groupoids (see, e.g., \cite{MM}), one could state this result in
purely finite dimensional terms, as a particular case of a tubular neighborhood theorem in this context. For this kind of general statement,
to make precise the meaning of ``compactness in the transversal direction'' one should use the notion of compact generation introduced by Haefliger in 
\cite{Hae}).
\end{remark}

Let us turn then to the proof of Proposition \ref{technical}. 
We will consider cross-sections of the fibration $p: Y\to B$ whose fibers are the leaves of
$\mathcal{F}_Y$. A cross-section $\sigma: U\to Y$, defined over an open set $U\subset B$, can 
be identified with its image $\sigma(U)\subset Y$, which is a transversal to $\mathcal{F}_Y$.  
Note that, due to our hypothesis, if $U$ is contractible, then cross-sections over $U$ do exist. 

Given a cross-section $\sigma: U\to Y$, by a \textbf{transversal tubular neighborhood} of $\sigma$ 
we mean:
\[
\xymatrix{ 
E\,\ar[dr]_r\ar@<-2 pt>@{^{(}->}^{\tilde{\sigma}}[rr]&& X\\
&U\ar[ur]_{\sigma}
}
\]
where $r: E\to U$ is a vector bundle and $\tilde{\sigma}:E\to X$ is an embedding defining 
a tubular neighborhood of $\sigma(U)$ in some transversal $T$ to $\mathcal{F}$ containing $\sigma(U)$.
Hence $\tilde{\sigma}|_{U}= \sigma$ and $\tilde{\sigma}(E)$ is an open subset of $T$. We will assume that the vector bundle comes equipped with a norm $||~||$. The proof of existence of transversal tubular neighborhoods can be found in \cite{Hirsch}. Similarly, one can talk about a
\textbf{transversal partial tubular neighborhood} of $\sigma$ (\emph{loc.~cit.}~pp.~109): in this case $\tilde{\sigma}$ is only defined on an open neighborhood of the zero-section in $E$. Any such transversal partial neighborhood contains a transversal tubular neighborhood (\emph{loc.~cit.}). By abuse of notation we write $\tilde{\sigma}:E\to T$ for a transversal partial tubular neighborhood, even if it is only defined in a open neighborhood of the zero section in $E$. Also, we have the following extension property which follows from general properties of tubular neighborhoods (see, e.g., Exercise 3, pp. 118 in \cite{Hirsch}).

\begin{lemma}
\label{lemma:partial:transversal} 
Let $\sigma: U\to Y$ be a cross-section, $V$ and $W$ opens in $U$
such that $\overline{V}\subset W\subset \overline{W}\subset U$ (where the closures are in $B$). 
Let also $T$ be a transversal to $\mathcal{F}$ containing $\sigma(U)$. 
Assume that $\tilde{\sigma}_{W}: E_{W}\to T$ is a transversal tubular neighborhood
of $\sigma|_{W}$. Then there exists a transversal tubular neighborhood $\tilde{\sigma}: E\to T$
of $\sigma$, defined on some vector bundle $E$ over $U$, such that $E_W|_{V}= E|_{V}$ (as vector
bundles) and $\tilde{\sigma}= \tilde{\sigma}_{W}$ on $E|_{V}$.
\end{lemma}

A \textbf{homotopy of two cross-sections} $\sigma_0, \sigma_1:U\to Y$ defined over the same open set $U\subset B$ is a smooth family $\{\sigma_{t}:t\in [0, 1]\}$ of cross sections over $U$ connecting $\sigma_0$ and $\sigma_1$. Since the fibration $p: Y\to B$ is locally trivial it follows that any two cross-sections over a contractible open set are homotopic. 

Let $\sigma=\{\sigma_t\}$ be a homotopy between two cross-sections $\sigma_0,\sigma_1: U\to Y$.  Two transversal partial tubular neighborhoods $\tilde{\sigma}_{i}:E\to X$ of $\sigma_i$ ($i\in \{0, 1\}$) are said to be \textbf{$\sigma$-compatible} if the map
\[ \tilde{\sigma}_0(e)\stackrel{h}{\longmapsto} \tilde{\sigma}_1(e)\]
(defined for $e\in E$ in the intersection of the domains of $\tilde{\sigma_i}$) has the following properties:
\begin{enumerate}[(a)]
\item $x$ and $h(x)$ are in the same leaf of $\mathcal{F}$ for all $x$;
\item the germ of $h$ at each point $\sigma_0(u)$, where $u\in U$, coincides with the holonomy germ
of the foliation $\mathcal{F}$ along the path $t\mapsto \sigma_t(u)$.
\end{enumerate}

\begin{lemma}
\label{lemma:homotopies:comptab} Let $\sigma_0, \sigma_1: U\to Y$ be two cross-sections over an open $U\subset B$
connected by a homotopy $\sigma= \{\sigma_t\}$. Let $\tilde{\sigma}_{0}: E\to X$ be a transversal 
partial tubular neighborhood above $\sigma_0$ and
let $T$ be a transversal to $\mathcal{F}$ containing $\sigma_1(U)$. Then, for any
$V\subset B$ open with $\overline{V}\subset U$, there exists
\begin{enumerate}[(i)]
\item an open subset $F\subset E|_{V}$ containing $V$ (so that $\tilde{\sigma}_{0}|_{F}$ is a transversal partial
tubular neighborhood of $\sigma_0|_{V}$);
\item a transversal partial tubular neighborhood $\tilde{\sigma}_1: F\to T$ 
of $\sigma_1|_{V}$;
\end{enumerate} 
such that $\tilde{\sigma}_{0}|_{F}$ and $\tilde{\sigma}_1$ are $\sigma|_{V}= \{\sigma_t|_{V}\}$-compatible. 
\end{lemma}

\begin{proof}
Fix $\sigma_0$, $\sigma_1$, $\tilde{\sigma}_0: E\to X$ and $T$ as in the statement.
As a temporary terminology, we say that an open subset $V\subset U$ is \emph{good} if $\overline{V}\subset U$ and the conclusion of the lemma holds for $V$. An open subset of a good open set is also good.

We first show that any $u\in U$ admits a good open neighborhood. Consider the holonomy transformation along the path $\sigma^u(t):=\sigma(t,u)$ from the transversal $\tilde{\sigma}_0(E)$ to the transversal $T$. This is the germ of a diffeomorphism $h_u$, defined in some neighborhood of $\sigma_0(u)$, which can be taken of the form $\tilde{\sigma}_0(F)$ for some open set $F\subset E$ containing $u$. Choosing $F$ a small enough open ball (relative to $||~||$) and setting $\tilde{\sigma}_1:= h_u\circ\tilde{\sigma}_0|_{F}$, we conclude that $V$ is a good open set.

Let $V$ be an arbitrary open set with $\overline{V}\subset U$. We can find a cover of $\overline{V}$ by good open sets, so we can extract a finite good subcover $\{U_i: 1\leq i\leq p\}$ of $\overline{V}$. We prove by induction on $p$ that, if this happens, then $V$ must be a good open set. Obviously, the result holds if $p=1$. For the induction step, assume the assertion is true for $p-1$ and assume that $\overline{V}$ is covered by $p$ good open sets $U_i\subset U$. Choose another cover $\{V_i\}$ of $V$ with $\overline{V}_i\subset U_i$. Then, by the induction hypothesis, $U_1=V_1$ and $U_2:= V_2\cup \ldots \cup V_p$ will be good open sets. Moreover, $\overline{V}\subset U_1\cup U_2$, so all that remains to show is the case $p=2$.

Let $U_1,U_2\subset U$ be good opens sets and $\overline{V}\subset U_1\cup U_2\subset U$. We need to show that $V$ is a good open set. Let $F_i\subset E$, $\tilde{\sigma}_i: F_i\to T$ be the associated transversal partial tubular neighborhoods. Consider also the induced maps $h_i: \tilde{\sigma}_0(F_i)\to \tilde{\sigma}_i(F_i)$.  We consider two new open sets $V_i$ such that $\overline{V}\subset V_1\cup V_2$ and $\overline{V}_i\subset U_i$.
Compactness of $\overline{V}_i$ shows that we can find $R>0$ such that:
\[ x\in E|_{V_i}, ||x||< R \Longrightarrow x\in F_i.\]
Due to the properties of $h_i$ (properties (a) and (b) above), we see that $h_1$ and $h_2$ 
coincide in a neighborhood of $\sigma_0(u)$ in $\sigma_0(U)$. Hence, choosing
eventually a smaller $R$, we may assume that 
\[ x\in E|_{V_1\cap V_2}, ||x||< R \Longrightarrow \tilde{\sigma}_1(x)= \tilde{\sigma}_1(x).\]
It follows that $\tilde{\sigma}_1$ and $\tilde{\sigma}_2$ will glue on
\[ F= \{ x\in E_{V_1\cup V_2}: ||x||< R\}\]
and the resulting transversal partial tubular neighborhood will have the desired 
properties so that $V$ is a good open set.
\end{proof}

For the next lemma, we introduce the following notation. A \textbf{$\mathcal{F}$-data} is a tuple
\begin{equation}
\label{system} 
\{U_i,\sigma_i,\tilde{\sigma}_i,\sigma_{(i,j)},E : 1\leq i, j\leq k\}
\end{equation}
consisting of the following:
\begin{enumerate}[(a)]
\item $\{U_i: i=1,\dots,k\}$ is a family of open sets in $B$ and $E$ is a vector bundle over 
$U=U_1\cup\cdots\cup U_k$.
\item $\sigma_i: U_i\to Y$ are cross-sections and $\sigma_{(i,j)}$ are homotopies between $\sigma_i|_{U_i\cap U_j}$ and $\sigma_j|_{U_i\cap U_j}$.
\item $\tilde{\sigma}_i: E|_{U_i}\to X$ are
transversal tubular neighborhoods over $\sigma_i$ such that
the restrictions of $\tilde{\sigma}_i$ and $\tilde{\sigma}_j$ to
$E|_{U_i\cap U_j}$ are $\sigma_{i, j}$-compatible for all $i$ and $j$.
\end{enumerate}
Assume now that:
\begin{enumerate}[(i)]
\item $U_{k+1}\subset B$ is another open set, $\sigma_{k+1}: U_{k+1}\to Y$ is a cross-section
above $U_{k+1}$ and $T$ is a transversal to $\mathcal{F}$ containing $\sigma_{k+1}(U_{k+1})$.
\item for each $1\leq i\leq k$ we have a homotopy $\sigma_{(i,k+1)}$ between 
$\sigma_i|_{U_i\cap U_{k+1}}$ and $\sigma_{k+1}|_{U_i\cap U_{k+1}}$.
\end{enumerate}
Let $V_i\subset B$ be open sets with
\[ \overline{V}_i\subset U_i \ \ (1\leq i\leq k+1)\]
and set $V= V_1\cup \ldots V_{k}$, $V'= V\cup V_{k+1}$. Then:

\begin{lemma}
\label{lemma:extend:F:data} 
Under the above assumptions, one can find a vector bundle $E'$ over $V'$ together with an
embedding $\phi: E'|_{V}\hookrightarrow E|_{V}$ of bundles, as well as a map $\tilde{\sigma}_{k+1}^{'}: E'|_{V_{k+1}}\to  X$ which is a transversal tubular neighborhood of $\sigma_{k+1}|_{V_{k+1}}$ inside $T$
such that
\[ \{V_i, \sigma_{i}^{'},\tilde{\sigma}_{i}\circ\phi,\sigma_{(i,j)}|_{V_i\cap V_j}, E^{'}: 1\leq i, j\leq k+1\}\]
is a $\mathcal{F}$-data. 
\end{lemma}

\begin{proof} 
For $1\leq i\leq k+1$, choose open sets $\overline{V}_i\subset V_{i}^{'}\subset \overline{V}_{i}^{'}\subset U_{i}$. We can apply Lemma \ref{lemma:homotopies:comptab} to:
\begin{itemize}
\item the restrictions of $\sigma_{i}$ and $\sigma_{k+1}$ to $U_i\cap U_{k+1}$ and the homotopy $\sigma_{(i,k+1)}$.
\item the transversal tubular neighborhood to $\sigma_{i}|_{U_i\cap U_{k+1}}$ which is the restriction 
of $\tilde{\sigma}_i$ to $E|_{U_i\cap U_{k+1}}$. 
\item the open set $V^{'}_{i}\cap V^{'}_{k+1}$ whose closure is inside $U_{i}\cap U_{k+1}$.
\end{itemize}
It gives a transversal tubular neighborhood of $\sigma_{k+1}|_{V_{i}^{'}\cap V_{k+1}^{'}}$, denoted
\[ \tilde{\sigma}_{k+1}^{(i)}: E_{i, k+1}^{'}\to T\]
defined on some $E_{i, k+1}^{'}\subset E|_{V_{i}^{'}\cap V_{k+1}^{'}}$, an open set containing $V_{i}^{'}\cap V_{k+1}^{'}$.

Now choose open sets $\overline{V}_i\subset V_{i}^{''}\subset \overline{V}_{i}^{''}\subset V_{i}^{'}$. 
Since the closure of $V_{i}^{''}\cap V_{k+1}^{''}$ is compact, we find $R_i>0$ such that
\[ x\in E|_{V_{i}^{''}\cap V_{k+1}^{''}}, ||x||< R_i\Longrightarrow x\in E_{i, k+1}^{'}.\]
Next, for each $1\leq i, j\leq k$, the restrictions of $\tilde{\sigma}_{k+1}^{(i)}$ and $\tilde{\sigma}_{k+1}^{(j)}$ to $E_{i, k+1}^{'}\cap E_{j, k+1}^{'}$ are transversal partial tubular neighborhoods above the same cross-section $\sigma_{k+1}|_{V_{i}^{'}\cap V_{j}^{'}}$. Moreover, they are
$\sigma$-compatible, where $\sigma$ is the concatenation of the homotopies $\sigma_{(i,k+1)}$,
$\sigma_{(j,i)}$ and $\sigma_{(k+1,j)}$. Since all paths $\sigma^u(-)= \sigma(u, -)$
induced by a homotopy $\sigma$ are inside leaves of $\mathcal{F}_Y$ and these leaves are assumed
to be simply connected, the holonomy germs induced by the closed loops $\sigma_{(k+1,j)}^u\circ\sigma_{(j,i)}^u\circ\sigma_{(i,k+1)}^u$ are trivial. We conclude that
\[ \{ x\in E_{i, k+1}^{'}\cap E_{j, k+1}^{'}: \tilde{\sigma}_{k+1}^{(i)}(x)= \tilde{\sigma}_{k+1}^{(j)}(x)\}\]
contains an open subset in $E_{i, k+1}^{'}\cap E_{j, k+1}^{'}$ containing $V_{i}^{'}\cap V_{j}^{'}\cap V_{k+1}^{'}$. Again, we can find constants $R_{i,j}$ such that:
\[ x\in E_{i,k+1}^{'}|_{V_{i}^{''}\cap V_{j}^{''}\cap V_{k+1}^{''}}\cap E_{j, k+1}^{'}|_{V_{i}^{''}\cap V_{j}^{''}\cap V_{k+1}^{''}}, ||x||< R_{i,j} \Longrightarrow \tilde{\sigma}_{k+1}^{(i)}(x)= \tilde{\sigma}_{k+1}^{(j)}(x).\]
Let us set:
\[ R= \textrm{min} \{R_i, R_{i, j}: 1\leq i, j\leq k\},\quad \ E^{''}= \{x\in E: ||x||< R\}.\]
The maps $\tilde{\sigma}_{k+1}^{i}$ glue together to give a smooth map defined on $E^{''}|_{V''\cap V_{k+1}^{''}}$ where
\[ V''= V_{1}^{''}\cup \ldots \cup V_{k}^{''} .\] 
Denote this map by $\tilde{\sigma}_{k+1}^{''}$. 
Consider now $\lambda: [0, \infty)\to [0, 1)$ be a diffeomorphism equal to the identity near $0$ and define the
embedding
\[ \phi : E\to E^{''}, h(x)= R\frac{\lambda(||x||)}{||x||} x.\]
Composing with this embedding, we obtain 
\[\tilde{\sigma}_{k+1}^{''}: E|_{V_{k+1}^{''}\cap V''}\to T\]
which is a transversal tubular neighborhood of $\sigma_{k+1}|_{V_{k+1}^{''}\cap V''}$. We can now apply Lemma \ref{lemma:partial:transversal} to find:
\begin{enumerate}[(i)]
\item a vector bundle $E_{k+1}$ over $V_{k+1}$ such that $E_{k+1}|_{V_{k+1}\cap V}=E|_{V_{k+1}\cap V}$.
\item a transversal tubular neighborhood $\tilde{\sigma}_{k+1}^{'}$ of
$\sigma_{k+1}|_{V_{k+1}}$ defined on the entire  $E_{k+1}$ and which 
coincides with $\tilde{\sigma}_{k+1}^{''}$ on $E|_{V_{k+1}\cap V}$.
\end{enumerate}
Finally, if we let $E'$ be the vector bundle over $V'= V\cup V_{k+1}$ obtained by gluing $E|_{V}$ (over $V$) and $E_{k+1}$ (over $V_{k+1}$), we have obtained the desired $\F$-data.
\end{proof}

\begin{proof}[Proof of Proposition \ref{technical}]
Let $\mathcal{U}^{(1)}= \{U_{1}^{(1)}, \ldots , U^{(1)}_{n}\}$ be a finite good cover of $B$ (since $B$ is compact, they exist; see \cite{Bott}). Since each $U^{(1)}_{i}$ is contractible, there are cross-sections $\sigma_i:U_i\to Y$. Since all intersections $U^{(1)}_{i}\cap U^{(1)}_{j}$ are contractible and the fibers of $p: Y\to B$ are 
contractible, there are homotopies $\sigma_{(i,j)}$ between $\sigma_i|_{U_i\cap U_j}$ and $\sigma_j|_{U_i\cap U_j}$.
Any cover can be refined by a finite good cover, so we may also assume that the image of each $\sigma_i$ is inside some transversal $T_i$ of $\mathcal{F}$. Finally, we choose new good covers $\mathcal{U}^{(k)}= \{U_{1}^{(k)}, \ldots , U^{(k)}_{n}\}$, $k=1,\dots,n$, with the property
\[ \overline{U}^{(k)}_{i}\subset U^{(k+1)}_{i}\quad (i,k=1,\dots,n).\]
We then apply inductively Lemma \ref{lemma:extend:F:data}: at each step one gets a vector bundle over
$U^{(k)}_{1}\cup \ldots \cup U^{(k)}_{k}$ and an $\mathcal{F}$-data. For $k=n$, we obtain a vector
bundle over $B$, a complete transversal to $\mathcal{F}_{Y}$ (the images of the $U^{(n)}_{i}$'s by the cross sections) and the transversal to $\mathcal{F}$ (the transversal tubular neighborhoods of the final $\mathcal{F}$-data). 

It only remains to show that $T_X/\Hol_{T_Y}(\mathcal{F}_Y)$ is Hausdorff manifold. This can be checked directly. Here we indicate a more conceptual argument which is based on general properties of groupoids and their representations (in the sense of spaces on which the groupoid act). For Morita equivalences, we refer to \cite{MM}.
First of all, representations can be transported along Morita equivalences and, provided the groupoids and the Morita equivalences used are Hausdorff, the Hausdorff property of representations is preserved by this transport. Secondly, since $\mathcal{F}_Y$ is induced by the submersion $p: Y\to B$, the groupoid $\Hol_{T_Y}(\mathcal{F}_Y)\tto T_Y$ is Morita equivalent to the trivial groupoid  $B\tto B$, via the bimodule $T_Y\stackrel{\text{id}}{\longleftarrow} T_Y \stackrel{p}{\longrightarrow} B$. Finally, one just remarks that under this equivalence, $T_X/\Hol_{T_Y}(\mathcal{F}_Y)$ is the representation of $B\tto B$ which corresponds to the representation $T_X$ of $\Hol_{T_Y}(\mathcal{F}_Y)$.
\end{proof}

\section*{Appendix 2: Historical remarks}

The study of the linearization problem for Poisson brackets was
initiated by Alan Weinstein in the foundational paper
\cite{Wein}. There, he states the problem and he shows that the
formal linearization problem can be reduced to a cohomology
obstruction. If the isotropy Lie algebra is semisimple this
obstruction vanishes. For analytic linearization he conjectured that,
provided the isotropy Lie algebra is semisimple, this can always be
achieved, a result later proved by Conn \cite{Conn1}.

In \cite{Wein}, Alan Weinstein also considers the smooth linearization
problem. He gives an example of a smooth, non-linearizable, Poisson
bracket with isotropy Lie algebra $\mathfrak{sl}(2,\Rr)$. The
situation is remarkable similar to the case of Lie algebra actions,
and this counter-example is analogous to the example of a
non-linearizable smooth action of $\mathfrak{sl}(2,\Rr)$, given by
Guillemin and Sternberg in \cite{GiSt}. By contrast, he suggests that
linearization when the isotropy is $\mathfrak{so}(3)$ could be proved
as follows (see \cite{Wein}, page 539):
\begin{quote}
\emph{
The first step would be to use the theorems of Reeb and Moussu
to ``linearize'' the foliation by symplectic areas. Next, a ``volume
preserving Morse lemma'' would be used to put in standard form the
function which measures the symplectic area of the leaves. Finally,
the deformation method of Moser and Weinstein would have to be applied
to each symplectic leaf, with care taken to assure regularity at the
origin.}
\end{quote}
The proof sketched was actually implemented by Dazord in \cite{Daz}.
Weinstein goes on to conjecture that smooth linearization can be achieved for compact semisimple isotropy. This again was proved to be so by Conn in \cite{Conn2}.

Conn's proof of smooth linearization is a highly non-trivial analytic
argument. He views the effect of changes of coordinates upon the Poisson
tensor as a non-linear partial differential operator. A combination of
Newton's method with smoothing operators, as devised by J.~Nash and
J.~Moser, is used to construct successive approximations to the desired
linearizing coordinates. The linearized equations that need to be
solved at each step are non-degenerate and overdetermined (the
operator differentiates only along the symplectic foliation). However,
by working at the level of Lie algebra cohomology of $\gg$ with
coefficients in the space of smooth functions on $\gg^*$, Conn is able
to find accurate solutions to the linearized equations. This involves
many estimates on the Sobolev norms, which are defined from the
 the Killing form, and so take advantage of its
invariance, non-degeneracy and definiteness.

After Conn's work was completed attention turned to other Lie
algebras. In \cite{Wein4}, Weinstein showed that semisimple Lie
algebras of real rank greater than one are non-linearizable, in
general. The case of real rank 1, with the exception of
$\mathfrak{sl}(2,\Rr)$, remains open. In \cite{Du}, Dufour studied
linearization when the isotropy belongs to a certain class of Lie
algebras, called non-resonant, which allowed him to classify all the
3-dimensional Lie algebras that entail linearizability. Dufour and
Zung proved formal and analytic linearization for the Lie algebra of
affine transformations $\mathfrak{aff}(n)$ \cite{DuZu}. There are also
examples of Poisson structures for which linearization can be decided
only from knowledge of its higher order jets (see
\cite{BaCr}). More recently, a Levi decomposition for Poisson brackets,
generalizing linearization, has been introduced by Wade (\cite{Wa},
formal category), Zung (\cite{Zun}, analytic category) and Zung and
Monnier (\cite{MonZun}, smooth category). The methods are entirely
similar to the ones of Weinstein and Conn. A survey of 
these results can be found in \cite{MonFer}.

In spite of Conn's master work, the question remained if a simple,
more geometric, proof of smooth linearization would be possible. In
the Introduction of \cite{Wein2}, Alan Weinstein writes:
\begin{quote}
\emph{Why is it so hard to prove the linearizability of Poisson structures
with semisimple linear part? Conn published proofs about 15 years ago
in a pair of papers full of elaborate estimates (\dots) no
simplification of Conn's proofs has appeared.}

\emph{This is a mystery to me, because analogous theorems about
linearizability of actions of semisimple groups near their fixed
points were proven (\dots) using a simple averaging.}
\end{quote}
In this paper he goes on to propose to use Lie algebroid/groupoid
theory to tackle this and other linearization problems. After this
work, it become clear that this would indeed be the proper setup for a
geometric proof of linearization. However, his attempt would not be
successful because some of the techniques needed were not available
yet. Some basic results on proper groupoids, as well as a full
understanding of the integrability problem for Poisson manifolds and
Lie algebroids was missing, and this was done later by us in
\cite{Cra,CrFe1,CrFe2}. 

In the end, the geometric proof we have given here, is really a
combination of classical results on Lie groups extended to the
groupoid context. Once the groupoid is brought into the picture, one
has the usual differential geometric machinery at hand, and hence also
all the standard techniques to deal with linearization problems one
finds in different contexts (Moser trick, Van Est argument, Reeb
stability, averaging). It is curious that the methods used are
so close to the proof suggested by Alan Weinstein for the case of
$\mathfrak{so}(3)$, that we have quoted above. A general setup to
discuss linearization problems and its relation to deformation
problems will be given elsewhere (work in progress).

Finally, note that it would be possible to combine our 
Proposition \ref{prop:integrable} with the linearization 
theorem for proper groupoids around fixed points (see \cite{Wein3,Zun1}), 
to obtain another proof of Conn's theorem (this would be a 
geometric-analytic proof, since the linearization of proper groupoids
also involves some estimates.)

\vskip 40 pt

\bibliographystyle{amsplain}

\begin{thebibliography}{11}
\bibitem{BaCr} J.~Basto-Gon\c{c}alves and I.~Cruz, Analytic
  $k$-linearizability of some resonant Poisson structures,
  \emph{Lett.~Math.~Phys.~}\textbf{49}  (1999), 59--66.

\bibitem{Bott} R.~Bott and L.W.~Tu, \emph{Differential forms in algebraic topology}, 
Graduate Texts in Mathematics, 82. Springer-Verlag, New York-Berlin, 1982.


\bibitem{Conn1} J.~Conn, Normal forms for analytic Poisson structures,
  \emph{Annals of Math.}~\textbf{119} (1984), 576--601.


\bibitem{Conn2} J.~Conn, Normal forms for smooth Poisson structures,
  \emph{Annals of Math.}~\textbf{121} (1985), 565--593.

\bibitem{Cra} M.~Crainic, Differentiable and algebroid cohomology, van
  Est isomorphisms, and characteristic classes,
  \emph{Comment.~Math.~Helv.~}\textbf{78} (2003), no. 4, 681--721.

\bibitem{CrFe1} M.~Crainic and R.L.~Fernandes, Integrability of Lie
  brackets, \emph{Ann.~of Math.~(2)} \textbf{157} (2003), 575--620.

\bibitem{CrFe2} M.~Crainic and R.L.~Fernandes, Integrability of Poisson
  brackets, \emph{J.~Differential Geom.}~\textbf{66} (2004), 71--137.


\bibitem{CrMo} M.~Crainic and I.~Moerdijk , Deformations of Lie brackets: cohomological aspects,
  \emph{Journal of European Mathematical Society~}\textbf{10} (2008), no. 4, pp. 1037--1059.


\bibitem{Daz} P.~Dazord, Stabilit\'e et lin\'earisation dans les
  vari\'et\'es de Poisson, \emph{in} Symplectic geometry and mechanics
  (Balaruc, 1983), 59--75, Travaux en Cours, Hermann, Paris, 1985.

\bibitem{Du} J.P.~Dufour, Lin\'earisation de certaines structures de
  Poisson, \emph{J.~Differential Geom.~}\textbf{32} (1990),
  415--428.

\bibitem{DuZu} J.P.~Dufour and N.T.~Zung, Nondegeneracy of the Lie
  algebra $\mathfrak{aff}(n)$ \emph{C.~R.~Acad.~Sci.~Paris
  S\'er.~I Math.~}\textbf{335} (2002), 1043--1046.

\bibitem{GiSt} V.~Guillemin and S.~Sternberg, Remarks on a paper of
  Hermann, \emph{Trans.~Amer.~Math.~Soc.~}\textbf{130} (1968),
  110--116.

\bibitem{GiWe} V.L.~Ginzburg and A.~Weinstein, Lie-Poisson structure
  on some Poisson Lie groups, \emph{J. Amer.~Math.~Soc.~}
  \textbf{5} (1992), 445--453.


\bibitem{Hae} A.~Haefliger, Foliations and compactly generated pseudogroups, in \emph{Foliations:~geometry~and~dynamics}
(Warsaw, 2000), World Sci. Publ., River Edge, NJ, 2002, 275--295.

\bibitem{Hirsch} M.W.~Hirsch, \emph{Differential topology}, Graduate Texts in Mathematics, 33. 
 Springer-Verlag, New York, 1994.

\bibitem{Ionut} I.~Marcut, Normal forms in Poisson geometry, \emph{work in progress} (PhD, Utrecht University).

\bibitem{MM} I.~Moerdijk and J.~Mrcun, \emph{Introduction to foliations and Lie groupoids}, 
Cambridge Studies in Advanced Mathematics, 91. Cambridge University Press, 2003.

\bibitem{MonFer} P.~Monnier and R.L.~Fernandes, Linearization of Poisson 
 brackets, \emph{Lett.~Math. Phys.~}\textbf{69} (2004), no.1, 89--114.

\bibitem{MonZun} P.~Monnier and N.T.~Zung, Levi decomposition for
  smooth Poisson structures, \emph{J.~Differential Geom.~}\textbf{68} 
  (2004), no.2, 347--395.

\bibitem{Wa} A.~Wade, Normalisation formelle de structures de Poisson, 
  \emph{C.~R.~Acad.~Sci.~Paris S\'er.~I Math.~}\textbf{324} (1997),
  531--536.

\bibitem{Wein} A.~Weinstein, The local structure of Poisson manifolds,
  \emph{J.~Differential Geom.~}\textbf{18} (1983), 523--557.

\bibitem{Wein4} A.~Weinstein, Poisson geometry of the principal series
  and nonlinearizable structures,
  \emph{J.~Differential Geom.~}\textbf{25} (1987), 55--73.

\bibitem{Wein2} A.~Weinstein, Linearization problems for Lie
  algebroids and Lie groupoids, \emph{Lett.~Math. Phys.~}\textbf{52}
  (2000), 93--102.

\bibitem{Wein3} A.~Weinstein, Linearization of regular proper
  groupoids, \emph{J.~Inst.~Math.~Jussieu} \textbf{1} (2002), no.3, 493--511.

\bibitem{Zun} N.T.~Zung, Levi decomposition of analytic Poisson structures
 and Lie algebroids, \emph{Topology}, \textbf{42} (2003), 1403--1420.

\bibitem{Zun1} N.T.~Zung, Proper Groupoids and Momentum Maps: Linearization,
  Affinity and Convexity, \emph{Ann. Sci. \'Ec. Norm. Sup\'er., S\'er. 4}, 
  \textbf{39}  (2006), no.5, 841--869. 
\end{thebibliography}
\def\lllll{}

\end{document}